\documentclass[11pt]{amsart}
\usepackage{amsmath}
\usepackage{amssymb}
\usepackage{pifont}
\usepackage{amsfonts}
\usepackage{amsthm}
\usepackage{enumerate}
\usepackage[mathscr]{eucal}
\usepackage[pdftex, bookmarksnumbered, bookmarksopen, colorlinks, citecolor=blue, linkcolor=blue]{hyperref}

\newtheorem{theorem}{Theorem}[section]
\newtheorem{lemma}[theorem]{Lemma}
\newtheorem{proposition}[theorem]{Proposition}
\newtheorem{corollary}[theorem]{Corollary}

\theoremstyle{definition}

\theoremstyle{remark}
\newtheorem{remark}[theorem]{Remark}

\numberwithin{equation}{section}

\def\R{{\mathbb R}}

\def\intslash{\rlap{\kern  .32em $\mspace {.5mu}\backslash$ }\int}
\def\qsl{{\rlap{\kern  .32em $\mspace {.5mu}\backslash$ }\int_{Q_x}}}

\def\emph#1{{\it #1 }}

\def\lc{\lesssim}

\def\alp{\alpha}

\def\del{\delta}             
\def\eps{\varepsilon}

            \def\Lam{\Lambda}

              \def\Om{\Omega}
\def\fr{\frac}

\newcommand{\Be}{\begin{equation}}
\newcommand{\Ee}{\end{equation}}
\newcommand{\Bes}{\begin{equation*}}
\newcommand{\Ees}{\end{equation*}}
\newcommand{\Bsp}{\begin{split}}
\newcommand{\Esp}{\end{split}}
\newcommand{\Bm}{\begin{multline}}
\newcommand{\Em}{\end{multline}}
\newcommand{\Bea}{\begin{eqnarray}}
\newcommand{\Eea}{\end{eqnarray}}
\newcommand{\Beas}{\begin{eqnarray*}}
\newcommand{\Eeas}{\end{eqnarray*}}
\newcommand{\Benu}{\begin{enumerate}}
\newcommand{\Eenu}{\end{enumerate}}
\newcommand{\Bi}{\begin{itemize}}
\newcommand{\Ei}{\end{itemize}}
\begin{document}

\title[Discrete Fourier restriction]{A note on the discrete Fourier restriction problem}

\author{Xudong Lai}
\address{Xudong Lai(Corresponding author): Institute for Advanced Study in Mathematics, Harbin Institute of Technology, Harbin, 150001, People's Republic of China.\endgraf
}
\email{xudonglai@mail.bnu.edu.cn}

\thanks{The work is supported by NSFC (No.11371057, No.11471033, No.11571160), SRFDP (No.20130003110003), the Fundamental Research Funds for the Central Universities (No.2014KJJCA10) and the China Scholarship Council (No. 201506040129).}

\author{Yong Ding}
\address{Yong Ding: School of Mathematical Sciences, Beijing Normal University, Beijing 100875, People's Republic of China }

\email{dingy@bnu.edu.cn}

\subjclass[2010]{42B05,11L07}



\keywords{Discrete Fourier restriction, exponential sums.}

\begin{abstract}
In this paper, we establish a general discrete Fourier restriction theorem. As an application, we make some progress on the discrete Fourier restriction associated with KdV equation.
\end{abstract}

\maketitle

\section{Introduction}
Recently, the Fourier restriction problem has been widely studied (for example see \cite{Ste93}, \cite{Tao04}, \cite{BCT06}, \cite{BG11}, \cite{BD15}). In this paper, we investigate the discrete Fourier restriction problems. Let us first see the discrete Fourier restriction associated with KdV equations. More precisely, we are going to seek the best constant $A_{p,N}$ satisfying
\Be
\sum_{|n|\leq N}|\hat{f}(n,n^3)|^2\leq A_{p,N}\|f\|_{L^{p'}(\mathbb{T}^2)}^2
\Ee
where $f$ is a periodic function on $\mathbb{T}^2$, $\hat{f}$ is the Fourier transform of $f$ on $\mathbb{T}^2$, i.e. $\hat{f}(\xi)=\int_{\mathbb{T}^2}e^{-2\pi ix\cdot \xi}f(x)dx$, $N$ is a sufficient large integer, $p\geq 2$ and $\fr{1}{p}+\fr{1}{p'}=1$.
For any $\eps>0$, Bourgain \cite{Bou93} showed that $A_{6,N}\leq N^\eps$. Later Hu and Li \cite{HL13} proved that $A_{p,N}\lc_{\eps} N^{1-\fr{8}{p}+\eps}$ for $p\geq14$.

Bourgain \cite{Bou93} and Hu and Li \cite{HL13} conjectured that
\begin{equation}
A_{p,N}\leq \left\{\begin{array}{cccc}
      &C_p&\ for &2\leq p<8, \\
      &C_{\eps,p}N^{1-\fr{8}{p}+\eps}&\ for &p\geq8.
     \end{array}\right.
\end{equation}
Clearly, $p=8$ is the critical number. In this paper, we will make a slight progress of this conjecture. We will show that $A_{p,N}\lc_{\eps}N^{1-\fr{8}{p}+\eps}$ for $p\geq12$.

It is easy to see that the study of $A_{p,N}$ is equivalent to the periodic Strichartz inequality associated with KdV equation:
\Be\label{e:13strikdv}
\Big\|\sum_{|n|\leq N}a_ne^{2\pi i(xn+tn^3)}\Big\|_{L_{x,t}^p(\mathbb{T}^2)}\leq K_{p,N}\Big(\sum_{|n|\leq N}|a_n|^2\Big)^{\fr{1}{2}}.
\Ee

In fact, we have $A_{p,N}\approx K_{p,N}^2$ by using the dual method. Later while considering the Cauchy problem of the fifth-order KdV-type equations, Hu and Li \cite{HL15} studied the following Strichartz inequality
\Be\label{e:13strikdvtype}
\Big\|\sum_{|n|\leq N}a_ne^{2\pi i(xn+tn^k)}\Big\|_{L_{x,t}^p(\mathbb{T}^2)}\leq \mathcal{K}_{p,N}\Big(\sum_{|n|\leq N}|a_n|^2\Big)^{\fr{1}{2}},
\Ee
where $k$ is a positive integer and $k\geq 2$.
They \cite{HL15} proved that $\mathcal{K}_{6,N}\lc N^{\eps}$ if $k$ is odd and
$\mathcal{K}_{p,N}\lc_{\eps}N^{\fr{1}{2}(1-\fr{2(k+1)}{p})+\eps}$
for $p\geq p_0$ where
\Bes
p_0=\left\{\begin{array}{cccc}
(k-2)2^k+6\ \ &\text{if $k$ is odd},\\
(k-1)2^k+4\ \ &\text{if $k$ is even}.
\end{array}
\right.
\Ees

In \eqref{e:13strikdv} and \eqref{e:13strikdvtype}, the discrete Fourier restriction problems are studied  in two dimensions when the Fourier transform is indeed restricted to the curve $(n,n^3)$ and $(n,n^k)$. It is natural to consider a similar problem for higher dimensions when the Fourier transform is restricted to the general curve $(n^{k_1},\cdots,n^{k_d})$, where $k_1,\cdots,k_d$ are positive integers. Let $K_{p,d,N}$ be the best constant in the following inequality
\Be\label{e:13kpdng}
\Big\|\sum_{|n|\leq N}a_ne^{2\pi i(\alp_1n^{k_1}+\cdots+\alp_dn^{k_d})}\Big\|_{L^p(\mathbb{T}^d)}\leq K_{p,d,N}\Big(\sum_{|n|\leq N}|a_n|^2\Big)^{\fr{1}{2}}.
\Ee

Our main result in the present paper is as follows.
\begin{theorem}\label{t:13}
Let $a_n$ be a complex number for all $|n|\leq N$. Let $d>1$ and $k_1,\cdots,k_d$ be positive integers with $1\leq k_1<\cdots<k_d=k$. Set $\mathfrak{K}=\sum_{i=1}^dk_i$. Let $K_{p,d,N}$ be defined in \eqref{e:13kpdng}. Suppose $p\geq k(k+1)$. Then for any $\eps>0$, we have
\Be\label{e:13kpdn}
K_{p,d,N}\lc_{\eps}N^{\fr{1}{2}(1-\fr{2\mathfrak{K}}{p})+\eps},
\Ee
where the implicit constant depends on $k_1,\cdots,k_d$, $p$, $\eps$, but does not depend on $N$.
\end{theorem}
\begin{remark}
In \cite{Woo16}, T.D. Wooley adapted the efficient congruencing method to prove that \eqref{e:13kpdn} holds for $p\geq2k(k+1)$. And whenever $p>2k(k+1)$, one may take $\eps=0$ in \eqref{e:13kpdn}.
\end{remark}
\begin{remark}
In Section \ref{s:133}, we will show the bound in \eqref{e:13kpdn} is sharp up to a constant $\eps$. One may conjecture \eqref{e:13kpdn} holds for all $p\geq 2\mathfrak{K}$. Notice if $k_i=i$, $i=1,\cdots,d$, then $2\mathfrak{K}=d(d+1)$. Thus \eqref{e:13kpdn} is valid for $p\geq 2\mathfrak{K}$ in this case.
\end{remark}
By using Theorem \ref{t:13}, one could make some progresses on the previous results. Applying Theorem \ref{t:13} with $d=2$, $k_1=1$, $k_2=3$ and $d=2$, $k_1=1$, $k_2=k$(here $k\geq2$), we may obtain the following corollaries.
\begin{corollary}
Let $K_{p,N}$ be defined in \eqref{e:13strikdv}. Suppose $p\geq 12$. Then for any $\eps>0$, we get
$$K_{p,N}\lc_{\eps} N^{\fr{1}{2}(1-\fr{8}{p})+\eps},$$
where the implicit constant is independent of $N$. If $p>24$, one may take $\eps=0$.
\end{corollary}

\begin{corollary}
Let $\mathcal{K}_{p,N}$ be defined in \eqref{e:13strikdvtype}. Suppose $p\geq k(k+1)$. Then for any $\eps>0$, we have
\Bes
\mathcal{K}_{p,N}\lc_{\eps} N^{\fr{1}{2}(1-\fr{2(k+1)}{p})+\eps},
\Ees
where the implicit constant is independent of $N$. If $p>2k(k+1)$, one may take $\eps=0$.
\end{corollary}

By setting $d=k$, $k_i=i$, for $i=1,\cdots,k$, $a_n=1$ for $n=1,\cdots,N$, $a_n=0$ for $n=0,-1,\cdots,-N$ in Theorem \ref{t:13}, one obtain
\Be\label{e:13vino}
\int_{\mathbb{T}^k}\Big|\sum_{n=1}^Ne^{2\pi i(\alp_1 n+\cdots+\alp_kn^k)}\Big|^pd\alp\lc_{\eps}N^{p-k(k+1)+\eps}
\Ee
for $p\geq k(k+1)$, which is the Vinogradov's mean value theorem proved by Bourgain, Demeter and Guth \cite{BDG16} recently. \eqref{e:13kpdng} can be regarded as a weighted version of \eqref{e:13vino} and \eqref{e:13kpdng} is apparently harder than \eqref{e:13vino}. Notice that the curve $(t^{{k_1}},t^{k_2},\cdots,t^{k_d})$ may be degenerate, for example the curve $(t,t^3)$ has zero curvature at point $(0,0)$. It seems to be difficulty to use the method developed in \cite{BD15} and \cite{BDG16} to prove \eqref{e:13kpdn} for $p\geq 2\mathfrak{K}$, since what they deal with are hypersurface with nonzero Gaussian curvature or nondegenerate curve. The proof of Theorem \ref{t:13} is based on a key lemma from \cite{BDG16}. Bourgain \emph{et al.} \cite{BDG16} used this lemma to prove \eqref{e:13vino}.

Throughout this paper, the letter $C$ stands for a positive constant and $C_a$ denotes a constant depending on $a$. $A\lc_{\eps} B$ means $A\leq C_{\eps}B$ for some constant $C_\eps$. $A\approx B$ means that $A\lc B$ and $B\lc A$. For a set $E\subset\R^d$, we denote Lebesgue measure of $E$ by  $|E|$.
\vskip1cm

\section {Proof of Theorem \ref{t:13}}\label{s:132}
Before giving the proof of Theorem \ref{t:13}, we first introduce some lemmas.

\begin{lemma}[see Theorem 4.1 in \cite{BDG16}]\label{l:13BDG}
For each $1\leq n\leq N$, let $t_n$ be a point in $(\fr{n-1}{N},\fr{n}{N}]$. Suppose $B_R$ is a ball in $\R^d$ with center $c_B$ and radius $R$. Define $w_{B_R}(x)=\Big(1+\fr{|x-c_B|}{R}\Big)^{-200}$. Then for each $R\gtrsim N^d$, each ball $B_R$ in $\R^d$, each $a_n\in \mathbb{C}$, each $p\geq2$ and $\eps>0$, we have
\begin{equation}
\begin{split}
\Big(\fr{1}{|B_R|}&\int\Big|\sum_{n=1}^Na_ne^{2\pi i(x_1t_n+\cdots+x_dt_n^d)}\Big|^pw_{B_R}(x)dx\Big)^{\fr{1}{p}}\\
&\lc_{\eps}\Big(N^\eps+N^{\fr{1}{2}(1-\fr{d(d+1)}{p})+\eps}\Big)\Big(\sum_{n=1}^ {N}|a_n|^2\Big)^{\fr{1}{2}},
\end{split}
\end{equation}
where the implicit constant does not depend on $N$, $R$ and $a_n$.
\end{lemma}

\begin{lemma}\label{l:13von}
Suppose $a_n\in\mathbb{C}$ and $p\geq2$. Then for any $\eps>0$, we get
\begin{equation}\label{e:13poly}
\begin{split}
\Big(\int_{\mathbb{T}^d}&\Big|\sum_{|n|\leq N}a_ne^{2\pi i(x_1n+x_2n^2+\cdots+x_dn^d)}\Big|^pdx\Big)^{\fr{1}{p}}\\
&\lc_{\eps}\Big(N^\eps+N^{\fr{1}{2}(1-\fr{d(d+1)}{p})+\eps}\Big)\Big(\sum_{|n|\leq N}|a_n|^2\Big)^{\fr{1}{2}},
\end{split}
\end{equation}
where the implicit constant is independent of $N$ and $a_n$.
\end{lemma}
\begin{proof}
We first notice that the function
$$\sum_{|n|\leq N}a_ne^{2\pi i(x_1n+x_2n^2+\cdots+x_dn^d)}$$
is periodic with period $1$ in the variables $x_1,\cdots,x_d$.
By using Minkowski's inequality, making a change of variables and the above periodic fact, one may get
\Bes
\begin{split}
&\ \ \Big(\int_{\mathbb{T}^d}\Big|\sum_{|n|\leq N}a_ne^{2\pi i(x_1n+x_2n^2+\cdots+x_dn^d)}\Big|^pdx\Big)^{\fr{1}{p}}\\
&\leq a_0+\Big(\int_{\mathbb{T}^d}\Big|\sum_{n=1}^{ N}a_ne^{2\pi i(x_1n+x_2n^2+\cdots+x_dn^d)}\Big|^pdx\Big)^{\fr{1}{p}}\\
&\ \ +\Big(\int_{\mathbb{T}^d}\Big|\sum_{n=1}^{ N}a_{-n}e^{2\pi i(x_1n+x_2n^2+\cdots+x_dn^d)}\Big|^pdx\Big)^{\fr{1}{p}}.
\end{split}
\Ees
Hence, to prove \eqref{e:13poly}, it suffices to show that
\Bes
\Big(\int_{\mathbb{T}^d}\Big|\sum_{n=1}^{ N}a_ne^{2\pi i(x_1n+x_2n^2+\cdots+x_dn^d)}\Big|^pdx\Big)^{\fr{1}{p}}
\Ees
has the desired bound.
Applying Lemma \ref{l:13BDG} with $R=\sqrt{d}N^d$, $t_n=\fr{n}{N}$ and $B_R=B(0,R)$ which is centred at $0$, we may obtain
\Be\label{e:13BR}
\begin{split}
&\Big(N^{-d^2}\int\Big|\sum_{n=1}^{ N}a_ne^{2\pi i(x_1\fr{n}{N}+\cdots+x_d(\fr{n}{N})^d)}\Big|^pw_{B_R}(x)dx\Big)^{\fr{1}{p}}\\
&\lc_{\eps}\Big(N^\eps+N^{\fr{1}{2}(1-\fr{d(d+1)}{p})+\eps}\Big)\Big(\sum_{n=1}^{N}|a_n|^2\Big)^{\fr{1}{2}},
\end{split}
\Ee
Since $w_{B_R}(x)\approx 1$ on $B(0,R)$ and $[0,N^d]^d\subset B(0,R)$, the left side of \eqref{e:13BR} is bigger than
$$\Big(N^{-d^2}\int_{[0,N^d]^d}\Big|\sum_{n=1}^{ N}a_ne^{2\pi i(x_1\fr{n}{N}+\cdots+x_d(\fr{n}{N})^d)}\Big|^pdx\Big)^{\fr{1}{p}}$$

By making a change of variables, $x_1=N\alp_1,\cdots,x_d=N^d\alp_d$, the above integral equals to
\Be\label{e:13alp}
\Big(N^{-d^2+\fr{d(d+1)}{2}}\int_{A_N}\Big|\sum_{n=1}^{ N}a_ne^{2\pi i(\alp_1{n}+\cdots+\alp_d{n}^d)}\Big|^pd\alp\Big)^{\fr{1}{p}}
\Ee
where $A_N=[0,N^{d-1}]\times[0,N^{d-2}]\times\cdots\times[0,1]$.
Notice that the function
$$K_N(\alp)=\sum_{n=1}^Na_ne^{2\pi i(\alp_1n+\alp_2n^2+\cdots+\alp_dn^d)}$$
is periodic with period $1$ in the variables $\alp_1,\cdots,\alp_d$. Since $A_N$ has $N^{\fr{d(d-1)}{2}}$ number of unit cubes,
by the periodic fact of $K_N(\alp)$, it follows that \eqref{e:13alp} is equal to
$$\Big(\int_{\mathbb{T}^d}\Big|\sum_{n=1}^{ N}a_ne^{2\pi i(\alp_1{n}+\cdots+\alp_d{n}^d)}\Big|^pd\alp\Big)^{\fr{1}{p}},
$$
which completes the proof.
\end{proof}

Now we begin with the proof of Theorem \ref{t:13}. We first show that the proof can be reduced to the case $p_k=k(k+1)$, that is
\Be\label{e:13cri}
\Big\|\sum_{|n|\leq N}a_ne^{2\pi i(\alp_1n^{k_1}+\cdots+\alp_dn^{k_d})}\Big\|_{L^{p_k}(\mathbb{T}^d)}\lc_{\eps}  N^{\fr{1}{2}(1-\fr{2\mathfrak{K}}{p_k})+\eps}\Big(\sum_{|n|\leq N}|a_n|^2\Big)^{\fr{1}{2}}.
\Ee
Suppose \eqref{e:13cri} is true. Utilizing Cauchy-Schwarz inequality, we get
$$\Big\|\sum_{|n|\leq N}a_ne^{2\pi i(\alp_1n^{k_1}+\cdots+\alp_dn^{k_d})}\Big\|_{L^{\infty}(\mathbb{T}^d)}\lc  N^{\fr{1}{2}}\Big(\sum_{|n|\leq N}|a_n|^2\Big)^{\fr{1}{2}},
$$
By using the Riesz-Thorin interpolation theorem (see for example \cite{Gra249}) to interpolate \eqref{e:13cri} and the above $L^\infty$ estimate, one could easily get the required bound of $L^p$ estimate for $p\geq k(k+1)$ in Theorem \ref{t:13}. Therefore it remains to show \eqref{e:13cri}. Consider positive integers $k_1,\cdots,k_d$ with $1\leq k_1<\cdots<k_d=k$ and denote by $\{l_1,\cdots,l_s\}$ the complement set of $\{k_1,\cdots,k_d\}$ in $\{1,2,\cdots,k\}$. Set $\mathfrak{K}=\sum_{n=1}^dk_n$. Then we may see
\Be
\label{e:13sum}\sum_{i=1}^sl_i=\fr{1}{2}k(k+1)-\mathfrak{K}.
\Ee
Note that $p_k=k(k+1)$ is an even integer, therefore we may set $p_k=2u$. By using the simple fact $\int_0^1e^{2\pi i xy}dy=\del(x)$, here $\del$ is a Dirac measure at $0$, we have
\Be\label{e:13orth1}
\begin{split}
&\ \ \Lam:=\int_{\mathbb{T}^d}\Big|\sum_{|n|\leq N}a_ne^{2\pi i(\alp_1n^{k_1}+\cdots+\alp_dn^{k_d})}\Big|^{2u}d\alp\\
&=\int_{\mathbb{T}^d}\Big(\sum_{|n|\leq N}a_ne^{2\pi i(\alp_1n^{k_1}+\cdots+\alp_dn^{k_d})}\cdot\sum_{|m|\leq N}\overline{a_m}e^{-2\pi i(\alp_1m^{k_1}+\cdots+\alp_dm^{k_d})}\Big)^ud\alp\\
&=\sum_{|n_1|,\cdots,|n_u|\leq N,}\sum_{|m_1|,\cdots,|m_u|\leq N}a_{n_1}\cdots
 a_{n_u}\overline{a_{m_1}}\cdots\overline{a_{m_u}}\\
 &\ \ \ \ \ \ \ \ \ \ \ \ \times\del\Big(\sum_{i=1}^u(n_i^{k_1}-m_i^{k_1})\Big)\cdots\del\Big(\sum_{i=1}^u(n_i^{k_d}-m_i^{k_d})\Big).
\end{split}
\Ee
Thus \eqref{e:13orth1} equals to the number of integral solution of the system of equations
\Be\label{e:13sysequ1}
\left\{
\begin{array}{cccc}
\sum\limits_{i=1}^u(n_i^{k_1}-m_i^{k_1}) =0\\
\cdots\\
\sum\limits_{i=1}^u(n_i^{k_d}-m_i^{k_d}) =0\\
|n_i|\leq N, |m_i|\leq N, i=1,\cdots,u,
\end{array}
\right.
\Ee
with each solution counted with weight $a_{n_1}\cdots
 a_{n_u}\overline{a_{m_1}}\cdots\overline{a_{m_u}}$.

For each solution $(n_1,\cdots,n_u,m_1,\cdots,m_u)$ of \eqref{e:13sysequ1},
there exist integers $h_j,j=1,\cdots,k$, such that
$(n_1,\cdots,n_u,m_1,\cdots,m_u)$ is an integral solution of the following system of equations
\Be\label{e:13sysequ2}
\left\{
\begin{array}{cccc}
\sum\limits_{i=1}^u(n_i-m_i) =h_1\\
\sum\limits_{i=1}^u(n_i^2-m_i^2) =h_2\\
\cdots\\
\sum\limits_{i=1}^u(n_i^{k}-m_i^{k}) =h_k\\
|n_i|\leq N, |m_i|\leq N, i=1,\cdots,u,
\end{array}
\right.
\Ee
where $h_j=0$ if $j=k_i$ for some $i=1,\cdots,d$. By the last condition of \eqref{e:13sysequ2}, it is easy to see that $|h_j|\leq 2uN^j$ for $j=1,\cdots,k$.

On the other hand, for each integral solution $(n_1,\cdots,n_u,m_1,\cdots,m_u)$ of \eqref{e:13sysequ2} with $|h_j|\leq 2uN^j$ for $j=1,\cdots,k$ and $h_j=0$ if $j=k_i$ for some $1\leq i\leq d$, $(n_1,\cdots,n_u,m_1,\cdots,m_u)$ is also an integral solution of \eqref{e:13sysequ1}. Now we define
$$\Lam(h)=\int_{\mathbb{T}^k}\Big|\sum_{|n|\leq N}a_ne^{2\pi i(\alp_1n+\alp_2n^2+\cdots+\alp_kn^k)}\Big|^{2u}e^{2\pi i(-\alp_1h_1-\cdots-\alp_kh_k)}d\alp.$$
By using orthogonality, the above term is equal to
\Bes
\begin{split}
\sum_{|n_1|,\cdots,|n_u|\leq N}&\sum_{|m_1|,\cdots,|m_u|\leq N}a_{n_1}\cdots
 a_{n_u}\overline{a_{m_1}}\cdots\overline{a_{m_u}}\\
&\times\del\Big(\sum_{i=1}^u(n_i-m_i)-h_1\Big)\cdots\del\Big(\sum_{i=1}^u(n_i^{k}-m_i^{k})-h_k\Big),
\end{split}
\Ees
which counts the number of integral solution of \eqref{e:13sysequ2} with each solution counted with weight $a_{n_1}\cdots
 a_{n_u}\overline{a_{m_1}}\cdots\overline{a_{m_u}}$.
Combining above arguments, we conclude that
$$\Lam=\sum_{|h_{l_1}|\leq2uN^{l_1}}\cdots\sum_{|h_{l_s}|\leq2uN^{l_s}}\Lam(h)$$
where $h$ in the sum also satisfies $h_j=0$ if $j=k_i$ for some $i=1,\cdots,d$. Obviously, $|\Lam(h)|\leq\Lam(0)$. Hence we obtain
\Bes
\begin{split}
|\Lam|\leq\sum_{|h_{l_1}|\leq 2uN^{l_1}}\cdots\sum_{|h_{l_s}|\leq 2uN^{l_s}}\Lam(0)&\leq (2u)^sN^{l_1+\cdots+l_s}\Lam(0)\\
&\lc N^{\fr{1}{2}k(k+1)-\mathfrak{K}}N^{p_k\eps}\Big(\sum_{|n|\leq N}|a_n|^2\Big)^{\fr{p_k}{2}},
\end{split}
\Ees
where in the last inequality we use \eqref{e:13sum} and apply Lemma \ref{l:13von} with $p=k(k+1)$. Hence we establish \eqref{e:13cri} which completes the proof of Theorem \ref{t:13}.
\vskip1cm

\section{Sharpness of Theorem \ref{t:13}}\label{s:133}

In this section, we show that $N^{\fr{1}{2}(1-\fr{2\mathfrak{K}}{p})}$ is the best upper bound for $K_{p,d,N}$ when $p\geq 2\mathfrak{K}$. Therefore Theorem \ref{t:13} is sharp up to a factor of $N^\eps$.
\begin{proposition}
Let $K_{p,d,N}$ be defined in \eqref{e:13kpdng}. Suppose $p$ is an even integer. Then there exist constants $C_1$, $C_2$ such that
$$K_{p,d,N}\geq\max\Big\{C_1,C_2N^{\fr{1}{2}(1-\fr{2\mathfrak{K}}{p})}\Big\}.$$
\end{proposition}
\begin{proof}
Set $p=2u$. Let $1\leq k_1<k_2<\cdots<k_d$ and $\mathfrak{K}=k_1+\cdots+k_d$. Define
$$\Lam(N,2u)=\int_{\mathbb{T}^d}\Big|\sum_{|n|\leq N}e^{2\pi i(\alp_1n^{k_1}+\cdots+\alp_dn^{k_d})}\Big|^{2u}d\alp.$$
By using orthogonality, $\Lam(N,2u)$ counts the number of integral solution of the following system of equations
\Be\label{e:13sysequ3}
\left\{
\begin{array}{cccc}
\sum\limits_{i=1}^u(n_i^{k_1}-m_i^{k_1}) =0,\\
\cdots\\
\sum\limits_{i=1}^u(n_i^{k_d}-m_i^{k_d}) =0,\\
|n_i|\leq N, |m_i|\leq N, i=1,\cdots,u,
\end{array}
\right.
\Ee
Notice that the system of equations \eqref{e:13sysequ3} has $(2N+1)^u$ number of trivial solutions. In fact, for each $(n_1,\cdots,n_u)$ with $|n_i|\leq N$, $i=1,2,\cdots,u$, one may choose $(m_1,\cdots,m_u)=(n_1,\cdots,n_u)$. Hence we have
\Be\label{e:13upsolu1}
\Lam(N,2u)\geq CN^{\fr{p}{2}}.
\Ee

Define the set $\Om_N$ as
$$\Om_N=\Big\{\alp\in\mathbb{T}^d:|\alp_i|\leq\fr{1}{8dN^{k_i}},i=1,\cdots,d\Big\}.$$
Then we have $|\Om_N|\approx N^{-\mathfrak{K}}$. If $\alp\in\Om_N$ and $|n|\leq N$, then
\Bes
\begin{split}
\Big|\sum_{|n|\leq N}e^{2\pi i(\alp_1n^{k_1}+\cdots+\alp_dn^{k_d})}\Big|&\geq\Big|Re\sum_{|n|\leq N}e^{2\pi i(\alp_1n^{k_1}+\cdots+\alp_dn^{k_d})}\Big|\\
&\geq\sum_{|n|\leq N}\cos(2\pi(\alp_1n^{k_1}+\cdots+\alp_dn^{k_d}))\geq CN.
\end{split}
\Ees
Now we conclude that
\Be\label{e:13upsolu2}
\Lam(N,2u)\geq\int_{\Om_N}\Big|\sum_{|n|\leq N}e^{2\pi i(\alp_1n^{k_1}+\cdots+\alp_dn^{k_d})}\Big|^{2u}d\alp\geq CN^{p}|\Om_N|\geq CN^{p-\mathfrak{K}}.
\Ee
Recall $K_{p,d,N}$ is the best constant for the following inequality
$$\Big\|\sum_{|n|\leq N}a_ne^{2\pi i(\alp_1n^{k_1}+\cdots+\alp_dn^{k_d})}\Big\|_{L^p(\mathbb{T}^d)}\leq K_{p,d,N}\Big(\sum_{|n|\leq N}|a_n|^2\Big)^{\fr{1}{2}}.
$$
Choosing $a_n=1$ for all $|n|\leq N$, then we have $K_{p,d,N}\geq N^{-\fr{1}{2}}(\Lam(N,p))^{\fr{1}{p}}$.
Combining the estimates \eqref{e:13upsolu1} and \eqref{e:13upsolu2}, we may get $$K_{p,d,N}\geq\max\Big\{C_1,C_2N^{\fr{1}{2}(1-\fr{2\mathfrak{K}}{p})}\Big\}$$ which completes the proof.
\end{proof}

\subsection*{Acknowledgement}
The authors would like to express their deep gratitude to the referee for his/her very careful reading and valuable suggestions.

\vskip1cm

\bibliographystyle{amsplain}

\end{document}